\documentclass[11pt,a4paper]{article}
\usepackage{amsmath,amssymb}
\usepackage{bm}
\usepackage{ascmac}
\usepackage[dvipdfmx]{graphicx}
\usepackage{enumerate}
\usepackage[left=3cm, right=3cm]{geometry}

\newcounter{Cntr}
\usepackage{amsthm}
\newtheorem{thm}{Theorem}[section]
\newtheorem{prop}[thm]{Proposition}
\newtheorem{lem}[thm]{Lemma}
\newtheorem{cnj}[thm]{Conjecture}

\newtheorem{prob}[thm]{Problem}
\newtheorem{obs}[thm]{Observation}
\newtheorem{cla}{Claim}[section]

\theoremstyle{definition}

\newtheorem*{prf3}{Proof of Theorem \ref{thm:main2}}

\newtheorem{case}{Case}

\usepackage{latexsym}
\def\qed{\hfill $\Box$}

\begin{document}

\title{Hamiltonian cycles in 2-tough $2K_2$-free graphs} 

\author{%
    Katsuhiro Ota\thanks{E-mail address: \texttt{ohta@math.keio.ac.jp}}
    \qquad Masahiro Sanka\thanks{E-mail address: \texttt{sankamasa@keio.jp}}\\ 
    Department of Mathematics, Keio University, \\
    3-14-1 Hiyoshi, Kohoku-ku, Yokohama 223-8522, Japan
} 

\date{}

\maketitle

\begin{abstract}
    A graph $G$ is called a $2K_2$-free graph if it does not contain $2K_2$ as an induced subgraph. 
    In 2014, Broersma, Patel and Pyatkin showed that every 25-tough $2K_2$-free graph on at least three vertices is Hamiltonian. 
    Recently, Shan improved this result by showing that 3-tough is sufficient instead of 25-tough. 
    In this paper, we show that every 2-tough $2K_2$-free graph on at least three vertices is Hamiltonian, which was conjectured by Gao and Pasechnik. 
\end{abstract}


\section{Introduction}\label{intro}

The toughness was introduced by Chv\'{a}tal \cite{bib:Chvatal} in 1973. 
Let $\omega(G)$ denote the number of components of a graph $G$. 
A graph $G$ is {\it $t$-tough} if $|S| \geq t \cdot \omega(G-S)$ for every subset $S \subset V(G)$ with $\omega(G-S)>1$. 
The {\it toughness} $t(G)$ is the maximum value of $t$ for which $G$ is $t$-tough, or is $\infty$ if $G$ is a complete graph. 
Hence if $G$ is not a complete graph, then
\[t(G)=\min\left\{\dfrac{|S|}{\omega(G-S)} \mid S \subset V(G) \text{ with }\omega(G-S)>1\right\}.\]
A graph $G$ is {\it Hamiltonian} if it contains a {\it Hamiltonian cycle}, i.e., a cycle containing all vertices of $G$. 
It is well known that every Hamiltonian graph is 1-tough. 
However, the converse does not hold.
The following conjecture, proposed by Chv\'{a}tal, is still open. 

\begin{cnj}[Chv\'{a}tal \cite{bib:Chvatal}]\label{cnj:Chvatal}
    There exists a constant $t_0$ such that every $t_0$-tough graph on at least three vertices is Hamiltonian.
\end{cnj}

Bauer, Broersma and Veldman \cite{bib:Bauer} showed that for any $t<\frac{9}{4}$ there exists a $t$-tough graph which is not Hamiltonian. 
So if Conjecture \ref{cnj:Chvatal} is true, then such a constant $t_0$ must be at least $\frac{9}{4}$.

On the other hand, it is known that a toughness condition gives a $k$-factor of a graph.
Given a graph $G$ and a positive integer $k$, a {\it $k$-factor} of $G$ is a spanning subgraph of $G$ in which every vertex has degree $k$. 
In particular, a 2-factor is a spanning subgraph in which every component is a cycle. 
Thus, a connected 2-factor is a Hamiltonian cycle. 
For the existence of $k$-factors, Enomoto, Jackson, Katerinis and Saito showed the following theorem. 

\begin{thm}[Enomoto et al.~\cite{bib:Enomoto}]
    For $k \geq 1$, every $k$-tough graph on $n$ vertices such that $n \geq k+1$ and $kn$ is even has a $k$-factor.
\end{thm}

Partial results related to Conjecture \ref{cnj:Chvatal} have been obtained in various restricted classes of graphs.
For details of known results on Conjecture \ref{cnj:Chvatal}, we refer the reader to the survey \cite{bib:Bauer(survey)}. 
In this paper, we deal with Conjecture \ref{cnj:Chvatal} in $2K_2$-free graphs. 
A graph is {\it $2K_2$-free} if it does not contain $2K_2$ as an induced subgraph, where $2K_2$ is the graph consisting of four vertices and two independent edges. 
In 2014, Broersma, Patel and Pyatkin proved the following theorem. 

\begin{thm}[Broersma et al.~\cite{bib:Broersma}]\label{thm:25-tough}
    Every $25$-tough $2K_2$-free graph on at least three vertices is Hamiltonian.
\end{thm}

Recently, Shan improved the required toughness in Theorem \ref{thm:25-tough} from 25 to 3. 

\begin{thm}[Shan \cite{bib:Shan}]\label{thm:Shan}
    Every $3$-tough $2K_2$-free graph on at least three vertices is Hamiltonian.
\end{thm}

The class of $2K_2$-free graphs is a superclass of split graphs. 
A graph $G$ is called a {\it split graph} if the vertex set of $G$ can be partitioned into a clique and an independent set. 
In 1996, Kratsch, Lehel and M\"{u}ller showed the following theorem. 

\begin{thm}[Kratsch et al.~\cite{bib:Kratsch}]\label{thm:split graph}
    Every $\frac{3}{2}$-tough split graph on at least three vertices is Hamiltonian.
\end{thm}

Chv\'{a}tal \cite{bib:Chvatal} constructed a sequence $\{G_l\}_{l=1}^\infty$ of split graphs having no $2$-factors and $t(G_l)=\frac{3l}{2l+1}$ for each positive integer $l$. 
Thus, $\frac{3}{2}$ is the best possible toughness implying split graphs to be Hamiltonian. 
Since every split graph is a $2K_2$-free graph, one cannot decrease the toughness implying $2K_2$-free graphs to be Hamiltonian below $\frac{3}{2}$.

Also, Bauer, Katona, Kratsch and Veldman \cite{bib:Bauer(2-factor)} showed that every $\frac{3}{2}$-tough 5-chordal graph on at least three vertices has a 2-factor, where a graph is called 5-chordal if every chordless cycle of it has length at most five.
Since every $2K_2$-free graph is 5-chordal, we get the following proposition.

\begin{prop}\label{thm:main1}
    Every $\frac{3}{2}$-tough $2K_2$-free graph on at least three vertices has a $2$-factor.
\end{prop}

Gao and Pasechnik \cite{bib:Mou} conjectured that every 2-tough $2K_2$-free graph on at least three vertices is Hamiltonian. 
Theorem \ref{thm:Shan}, due to Shan \cite{bib:Shan}, is a partial solution to this conjecture.
In this paper, we prove this conjecture by modifying the proof in \cite{bib:Shan}.

\begin{thm}\label{thm:main2}
    Every $2$-tough $2K_2$-free graph on at least three vertices is Hamiltonian.
\end{thm}

It is unknown whether the toughness condition of Theorem \ref{thm:main2} is sharp. 
Considering Proposition \ref{thm:main1}, we conjecture that the following problem has a positive solution.

\begin{prob}
    Is every $\frac{3}{2}$-tough $2K_2$-free graph on at least three vertices Hamiltonian?
\end{prob}

\section{Terminology}\label{sec:term}

In this paper, we consider only undirected, finite and simple graph. The terminology not defined here can be found in \cite{bib:Diestel}.

Let $G$ be a graph and let $x \in V(G)$. 
The set of neighbors of $x$ in $G$ is denoted by $N_G(x)$ and the degree of $x$ in $G$ is denoted by $d_G(x)$. 
Let $S \subset V(G)$. 
We define the neighbors of $S$ by $N_G(S)=\bigcup_{x \in S}N_G(x)$. 
Let $H$ be a subgraph of $G$. 
Then we define that $V_H(x)=N_G(x) \cap V(H)$, $\overline{V}_H(x)=V(H) \setminus V_H(x)$, $V_H(S)=N_G(S) \cap V(H)$ and $\overline{V}_H(S)=V(H) \setminus V_H(S)$. 
For $xy \in E(G)$, $V_H(xy)$ and $\overline{V}_H(xy)$ denote $V_H(\{x,y\})$ and $\overline{V}_H(\{x,y\})$, respectively. 
Note that $N_H(S)$ is not equal to $V_H(S)$ in general.

Let $V_1, V_2 \subset V(G)$. 
Then $E_G(V_1,V_2)$ is the set of edges of $G$ with one end in $V_1$ and the other end in $V_2$. 
If $V_1=\{x\}$, then we write $E_G(x,V_2)$ for $E_G(\{x\},V_2)$. 
Moreover, for a subgraph $H$ of $G$, we write $E_G(V_1,H)$ and $E_G(H,V_2)$ for $E_G(V_1,V(H))$ and $E_G(V(H),V_2)$, respectively.

Let $P$ be a path. 
If $u$ and $v$ are the end-vertices of $P$, we say that $P$ is a $uv$-path. 
For $x,y \in V(P)$, $xPy$ denotes the path between $x$ and $y$ passing through $P$. 
Let $P$ be an $xy$-path and let $Q$ be a $uv$-path. 
If $P$ and $Q$ are disjoint, then $xPyuQv$ denotes the path between $x$ and $v$ passing through $P$, an added edge $yu$ and $Q$.

In this paper, we always assume that each cycle in a graph has a fixed orientation. 
Let $C$ be a cycle. 
For $x \in V(C)$, denote the successor of $x$ by $x^+$ and the predecessor of $x$ by $x^-$. 
Let $S \subset V(C)$. 
Then $S^+=\{x^+ \in V(C) \mid x \in S\}$ and $S^-$ is defined similarly. 
Let $D$ be another cycle disjoint from $C$ and $T \subset V(D)$. 
Then $(S \cup T)^+=S^+ \cup T^+$ and $(S \cup T)^-=S^- \cup T^-$. 
For $u,v \in V(C)$, $u\overrightarrow{C}v$ denotes the $uv$-path from $u$ to $v$ along the orientation of $C$. 
Also, $u\overleftarrow{C}v$ denotes $v\overrightarrow{C}u$. 
Let $P$ be an $xy$-path and let $Q$ be a $uv$-path. 
If $P$ and $Q$ are disjoint, then $xPyuQvx$ denotes the cycle passing through $P$, an added edge $yu$, $Q$ and an added edge $vx$.

\section{Proof of Theorem \ref{thm:main2}}
\setcounter{case}{0}

Our proof of Theorem \ref{thm:main2} basically follows the one in Shan \cite{bib:Shan}. 
Shan noted that the property of being 3-tough is used just once for proving one of the claims \cite[Claim 2.5]{bib:Shan}. 
Here we also improve the argument by introducing some new claims.

First, we introduce the following observation which immediately follows from the definition of $2K_2$-free graphs. 

\begin{obs}\label{prop:2K_2-free}
    For a graph $G$, the following statements are equivalent.
    \begin{enumerate}[$(1)$]
        \item $G$ is a $2K_2$-free graph.
        \item For every edge $xy \in E(G)$, the set $\overline{V}_G(xy)=V(G) \setminus (N_G(x) \cup N_G(y))$ is independent in $G$.
    \end{enumerate}
\end{obs}

For a graph $G$, let $\alpha(G)$ be the independence number of $G$. 
We show the following property of $2K_2$-free graphs.

\begin{lem}\label{lem:property of 2K_2-free}
    Let $G$ be a $2K_2$-free graph on $n$ vertices. 
    Then the set 
    \[\left\{x \in V(G) \mid d_G(x) \leq \frac{n-\alpha(G)}{2}\right\}\]
    is independent in $G$.
\end{lem}
\begin{proof}
    Suppose that there exists an edge $xy \in E(G)$ with $d_G(x),d_G(y) \leq \frac{n-\alpha(G)}{2}$. 
    By Observation \ref{prop:2K_2-free}, the set $\overline{V}_G(xy)$ is independent in $G$. 
    Since any vertex in $\overline{V}_G(xy)$ is not adjacent to $x$, the set $\overline{V}_G(xy) \cup \{x\}$ is independent in $G$. 
    However, since 
    \[|\overline{V}_G(xy)| \geq n-|N_G(x)|-|N_G(y)|=n-d_G(x)-d_G(y) \geq \alpha(G),\]
    we have $\alpha(G) \geq |\overline{V}_G(xy) \cup \{x\}| \geq \alpha(G)+1$, a contradiction.
\end{proof}

Next, we consider a 2-factor $F$ of a graph with minimum number of components. 
A vertex $x \in V(G)$ is said to be {\it co-absorbable} if there exists a 2-factor $F'$ of $G-x$ such that $\omega(F')<\omega(F)$.
Recall that each cycle in a graph has a fixed orientation.

\begin{lem}\label{lem:co-absorbable vtx}
    Let $G$ be a graph containing a $2$-factor, and let $F$ be a fixed $2$-factor of $G$ with minimum number of components. 
    If $x \in V(G)$ is co-absorbable, then $d_G(x) \leq \alpha(G)-1$.
\end{lem}
\begin{proof}
    Let $x \in V(G)$ be co-absorbable. 
    Then there exists a 2-factor $F'$ of $G-x$ such that $\omega(F')<\omega(F)$. 
    For each $y \in V(G) \setminus \{x\}$, in this proof, $y^+$ denotes the successor of $y$ in the unique cycle of $F'$ containing $y$. 
    Let $N_G(x)^+=\{y^+ \mid y \in N_G(x)\}$. 
    If $N_G(x)^+ \cup \{x\}$ is not independent, then we can easily find a 2-factor $\tilde{F}$ of $G$ such that $\omega(\tilde{F}) \leq \omega(F')<\omega(F)$, which contradicts the minimality of $\omega(F)$. 
    Thus, the set $N_G(x)^+ \cup \{x\}$ is independent in $G$. 
    This implies $d_G(x)=|N_G(x)^+| \leq \alpha(G)-1$. 
\end{proof}

\begin{prf3}

    Let $G$ be a 2-tough $2K_2$-free graph with at least three vertices. 
    Let $n=|V(G)|$. 
    Since the graph $G$ is 2-tough, $\alpha(G) \leq \frac{n}{3}$ and $G$ has a 2-factor. 
    We take a 2-factor $F$ of $G$ with minimum number of components. 
    Let $\mathcal{F}$ be the set of all cycles of $F$. 
    If $|\mathcal{F}|=1$, then $F$ is a Hamiltonian cycle of $G$. 
    We so assume that $|\mathcal{F}|>1$. 
    For $x \in V(G)$, in the following, $x^+$ and $x^-$ denotes the successor of $x$ and the predecessor of $x$ in the unique cycle of $F$ containing $x$, respectively. The following claim can be easily shown by the minimality of $|\mathcal{F}|$.

    \begin{cla}\label{edge_lemma}
        Let $C,D \in \mathcal{F}$ be two distinct cycles. 
        If $x \in V(C)$ and $y \in V(D)$ are adjacent in $G$, then $E_G(\{x^-,x^+\},\{y^-,y^+\})=\emptyset$.
    \end{cla}

    Let $x \in V(G)$ and let $C \in \mathcal{F}$ be the unique cycle such that $x \in V(C)$. 
    If there exists a cycle $D \in \mathcal{F} \setminus \{C\}$ such that $x$ is adjacent to two consecutive vertices on $D$ in $G$, we say that $x$ is of {\it $A$-type} (w.r.t.\ $D$). 
    If $x$ is not of $A$-type (w.r.t.\ any cycle in $\mathcal{F} \setminus \{C\}$), we say that $x$ is of {\it $B$-type}. 
    We define
    \[A=\{x \in V(G) \mid x \text{ is of } A \text{-type}\} \text{ and } B=V(G) \setminus A.\]
    For $I \subset V(G)$, we say that $I$ is {\it co-absorbable} if every vertex in $I$ is co-absorbable.

    \begin{cla}\label{co-abso}
        If $I \subset V(G)$ is co-absorbable, then $I$ is independent in $G$.
    \end{cla}
    \begin{proof}
        Let $I \subset V(G)$ be co-absorbable. 
        By Lemma \ref{lem:co-absorbable vtx}, for every $x \in I$, we have 
        \[d_G(x) \leq \alpha(G)-1<\frac{n-\alpha(G)}{2},\] 
        since $\alpha(G) \leq \frac{n}{3}$. 
        Thus, $I$ is independent in $G$ by Lemma \ref{lem:property of 2K_2-free}.
    \end{proof}

    \begin{cla}\label{B-type-co-abso}
        Let $C,D \in \mathcal{F}$ be two distinct cycles, and let $x \in V(C)$ and $y \in V(D)$ be two vertices such that $xy \in E(G)$. 
        If $x$ is of $B$-type and $y$ is co-absorbable, then $xy^+ \notin E(G)$ and $xy^{++} \in E(G)$.
    \end{cla}
    \begin{proof}
        We focus on two edges $xx^+$ and $y^+y^{++}$. 
        Since $x$ is of $B$-type, we have  $xy^+ \notin E(G)$. 
        Also, by Claim \ref{edge_lemma}, $x^+y^+ \notin E(G)$. 
        If $x^+y^{++} \in E(G)$, then 
        \[xy\overleftarrow{D}y^{++}x^+\overrightarrow{C}x\]
        is a cycle on $(V(C) \cup V(D)) \setminus \{y^+\}$, which implies that $y^+$ is co-absorbable. 
        Then, two co-absorbable vertices $y$ and $y^+$ are adjacent in $G$, contrary to Claim \ref{co-abso}. 
        Thus, $x^+y^{++} \notin E(G)$. 
        Since $\{x,x^+,y^+,y^{++}\}$ does not induce $2K_2$, we have $xy^{++} \in E(G)$.
    \end{proof}

    \begin{cla}\label{A-type}
        If $x \in V(G)$ is of $A$-type, then $x^+$ and $x^-$ are co-absorbable, and $x$ is not co-absorbable. 
    \end{cla}
    \begin{proof}
        Suppose that $C \in \mathcal{F}$ and $x \in A \cap V(C)$. 
        Then, we have a cycle $D \in \mathcal{F} \setminus \{C\}$ and $y \in V(D)$ such that $xy,xy^+ \in E(G)$. 
        By Claim \ref{edge_lemma}, $x^+y,x^+y^+ \notin E(G)$. 
        Considering two edges $x^+x^{++}$ and $yy^+$, since $G$ is $2K_2$-free, $x^{++}y \in E(G)$ or $x^{++}y^+ \in E(G)$. 
        Therefore,
        \begin{eqnarray*}
            \left \{
            \begin{aligned}
                &xy^+\overrightarrow{D}yx^{++}\overrightarrow{C}x \ &\text{if}& \ x^{++}y \in E(G),\\
                &xy\overleftarrow{D}y^+x^{++}\overrightarrow{C}x \ &\text{if}& \ x^{++}y^+ \in E(G)
            \end{aligned}
            \right.
        \end{eqnarray*}
        is a cycle on $(V(C) \cup V(D)) \setminus \{x^+\}$. 
        Thus, $x^+$ is co-absorbable. 
        We similarly find that $x^-$ is co-absorbable. 
        Since co-absorbable vertices are independent by Claim \ref{co-abso}, we conclude that $x$ is not co-absorbable.
    \end{proof}

    Let $C \in \mathcal{F}$ and $xy \in E(C)$. 
    By Claims \ref{co-abso} and \ref{A-type}, $x$ or $y$ is of $B$-type. 
    If both $x$ and $y$ are of $B$-type, we say that the edge $xy$ is of {\it $B$-type}; otherwise, $xy$ is of {\it $AB$-type}. 
    If all edges of $C$ are of $AB$-type, we say that $C$ is {\it $AB$-alternating}. 
    More generally, if vertices on $C$ are alternating between two disjoint sets $X \subset V(G)$ and $Y \subset V(G)$, we say that $C$ is {\it $XY$-alternating}.

    \begin{cla}\label{B_edge}
        Let $C \in \mathcal{F}$ and $xy \in E(C)$. 
        If $xy$ is of $B$-type, then every $D \in \mathcal{F} \setminus \{C\}$ is $V_D(xy)\overline{V}_D(xy)$-alternating.
    \end{cla}
    \begin{proof}
        It suffices to show that for every $D \in \mathcal{F} \setminus \{C\}$ and $uv \in E(D)$, exactly one vertex of $\{u,v\}$ is in $V_D(xy)$. 
        Since $G$ is $2K_2$-free, one of $\{u,v\}$ must be in $V_D(xy)$. 
        Suppose without loss of generality, that $u \in V_D(xy)$ with $ux \in E(G)$. 
        Then by Claim \ref{edge_lemma}, we have $vy \notin E(G)$. 
        As $x$ is of $B$-type and $ux \in E(G)$, we further have $vx \notin E(G)$. 
        Thus, $v \in \overline{V}_D(xy)$.
    \end{proof}

    \begin{cla}\label{cyc_cont_B_edge}
        Each of the following holds.
        \begin{enumerate}[$(1)$]
            \item $\mathcal{F}$ has a cycle containing a $B$-type edge.
            \item If $C \in \mathcal{F}$ contains a $B$-type edge, then $|V(C)| \geq \frac{n}{3}+2$.
        \end{enumerate}
    \end{cla}
    \begin{proof}
        By Claims \ref{co-abso} and \ref{A-type}, $A^+$ is independent in $G$ and so $|A|=|A^+| \leq \frac{n}{3}$, which implies (1). 
        For (2), suppose that $C \in \mathcal{F}$ contains a $B$-type edge $xy \in E(C)$. 
        By Claim \ref{B_edge}, we have
        \[|\overline{V}_{G-C}(xy)|=\sum_{D \in \mathcal{F} \setminus \{C\}}|\overline{V}_D(xy)|=\frac{n-|V(C)|}{2}.\]
        Since $\overline{V}_{G-C}(xy) \cup \{x\}$ is independent in $G$, we find that
        \[\frac{n-|V(C)|}{2}+1=|\overline{V}_{G-C}(xy) \cup \{x\}| \leq \frac{n}{3},\]
        which implies that $|V(C)| \geq \frac{n}{3}+2$.
    \end{proof}

    Let $\mathcal{F}_{AB}=\{C \in \mathcal{F} \mid \text{$C$ is $AB$-alternating}\}$ and $\overline{\mathcal{F}}_{AB}=\mathcal{F} \setminus \mathcal{F}_{AB}$. 
    By Claim \ref{cyc_cont_B_edge}, we have $|\overline{\mathcal{F}}_{AB}| \in \{1,2\}$. 
    Let $H \in \overline{\mathcal{F}}_{AB}$ and $x \in V(H) \cap B$.  
    If there exists $K \in \mathcal{F}_{AB}$ such that $V_K(x)=B \cap V(K)$, we say that $x$ is {\it bad w.r.t.\ $K$}. 
    For $H \in \overline{\mathcal{F}}_{AB}$, we define the set $V_{bad}(H)$ to be
    \[V_{bad}(H)=(V(H) \cap A) \cup \{x \in V(H) \cap B \mid \text{$x$ is bad w.r.t.\ some $K \in \mathcal{F}_{AB}$}\}.\]
    Let $A_0=\bigcup_{K \in \mathcal{F}_{AB}}(V(K) \cap A)$ and $B_0=\bigcup_{K \in \mathcal{F}_{AB}}(V(K) \cap B)$. 
    Note that $B_0=A_0^+ \subset A^+$.

    \begin{cla}\label{bad_set_lemma}
        Let $H \in \overline{\mathcal{F}}_{AB}$ and $x \in V(H)$. 
        If $E_G(x,B_0) \neq \emptyset$, Then $x \in V_{bad}(H)$.
    \end{cla}
    \begin{proof}
        Suppose that $E_G(x,B_0) \neq \emptyset$. 
        Then we can take $K \in \mathcal{F}_{AB}$ and $y \in V(K) \cap B$ such that $xy \in E(G)$. 
        If $x$ is of $A$-type, then $x \in V_{bad}(H)$ by the definition. 
        We may so assume that $x$ is of $B$-type. 
        Since $B \cap V(K)=A^+ \cap V(K)$, $y$ is co-absorbable. 
        Thus, $xy^+ \notin E(G)$ and $xy^{++} \in E(G)$ by Claim \ref{B-type-co-abso}. 
        By the repetition of this argument, we find that $V_K(x)=B \cap V(K)$, and hence $x$ is bad w.r.t.\ $K$. 
    \end{proof}

    \begin{cla}\label{bad_set_co-abso}
        Let $H \in \overline{\mathcal{F}}_{AB}$. 
        If $x \in V(H)$ is bad w.r.t.\ some $K \in \mathcal{F}_{AB}$, then $x^+$ and $x^-$ are co-absorbable. 
    \end{cla}
    \begin{proof}
        Suppose that $x \in V(H)$ is bad w.r.t.\ $K \in \mathcal{F}_{AB}$. 
        Now we have $V_K(x)=B \cap V(K)$. 
        Let $y \in B \cap V(K)$. 
        We focus on two edges $x^+x^{++}$ and $yy^+$. 
        Since $y$ is of $B$-type, $x^+y \notin E(G)$. 
        Moreover, $x^+y^+ \notin E(G)$ by Claim \ref{edge_lemma}. 
        Hence $x^{++}y \in E(G)$ or $x^{++}y^+ \in E(G)$ since $G$ is $2K_2$-free. 
        If $x^{++}y^+ \in E(G)$, then 
        \[xy\overleftarrow{K}y^+x^{++}\overrightarrow{H}x\] 
        is a cycle on $(V(H) \cup V(K)) \setminus \{x^+\}$. 
        Thus, $x^+$ is co-absorbable. 
        Next suppose that $x^{++}y \in E(G)$ and $x^{++}y^+ \notin E(G)$. 
        Since $y^+$ is of $A$-type, we can take a cycle $Q \in \mathcal{F} \setminus \{K\}$ and $z \in V(Q)$ such that $y^+z,y^+z^+ \in E(G)$. 
        Note that $z \notin \{x^-,x,x^+,x^{++}\}$ because $x^-y^+,x^+y^+,x^{++}y^+ \notin E(G)$. 
        Note that $yz,yz^+ \notin E(G)$ by Claim \ref{edge_lemma}, and hence we have $y^-z \in E(G)$ or $y^-z^+ \in E(G)$ since $G$ is $2K_2$-free. 
        Thus, considering one or two cycles
        \begin{eqnarray*}
            \left \{
                \begin{aligned}
                    &xyx^{++}\overrightarrow{H}zy^-\overleftarrow{K}y^+z^+\overrightarrow{H}x \ &\text{if}& \ H=Q \text{ and }y^-z \in E(G),\\
                    &xyx^{++}\overrightarrow{H}zy^+\overrightarrow{K}y^-z^+\overrightarrow{H}x \ &\text{if}& \ H=Q \text{ and }y^-z^+ \in E(G),\\
                    &xyx^{++}\overrightarrow{H}x, \ zy^-\overleftarrow{K}y^+z^+\overrightarrow{Q}z \ &\text{if}& \ H \neq Q \text{ and }y^-z \in E(G),\\
                    &xyx^{++}\overrightarrow{H}x, \ zy^+\overrightarrow{K}y^-z^+\overrightarrow{Q}z \ &\text{if}& \ H \neq Q \text{ and }y^-z^+ \in E(G),\\
                \end{aligned}
            \right.
        \end{eqnarray*}
        we find that $x^+$ is co-absorbable. 
        We similarly find that $x^-$ is co-absorbable.
    \end{proof}

    \begin{case}
        $|\overline{\mathcal{F}}_{AB}|=1$.
    \end{case}

    Let $C$ be the unique cycle in $\overline{\mathcal{F}}_{AB}$. 
    Now we have $\mathcal{F}_{AB} \neq \emptyset$ since $|\mathcal{F}|>1$. 
    If $V_{bad}(C)=\emptyset$, then $E_G(C,B_0)=\emptyset$ by Claim \ref{bad_set_lemma}.
    However, then every vertex in $B_0$ is an isolated vertex in $G-A_0$, a contradiction to the toughness condition of $G$. 
    Thus, we have $V_{bad}(C) \neq \emptyset$.

    For a vertex $x \in V_{bad}(C)$, we define
    \[U_x^0=\{x^+\} \text{ and } U_x^1=\{y \in V(C) \mid y^+ \in V_C(U_x^0) \setminus V_{bad}(C)\} \setminus U_x^0.\]
    For each vertex $x_1 \in U_x^1$, define the path
    \[P_{[x_1,x]}=x_1\overleftarrow{C}x^+x_1^+\overrightarrow{C}x\]
    to be the directed path from $x_1$ to $x$. 
    In general, for $i \geq 2$, we define
    \[U_x^i=\{u \mid u^\dagger v \in E(G), \text{ for some }v \in U_x^{i-1}\text{ and }u^\dagger \in V(C) \setminus V_{bad}(C)\} \setminus \bigcup^{i-1}_{j=0}U_x^j\]
    where $u^\dagger$ is the immediate successor of $u$ on $P_{[v,x]}$. 
    For each $u \in U^i_x$, we choose and fix $v \in U^{i-1}_x$ such that $u^\dagger v \in E(G)$ and $u^\dagger \notin V_{bad}(C)$ on $P_{[v,x]}$, and define the path
    \[P_{[u,x]}=uP_{[v,x]}vu^\dagger P_{[v,x]}x\]
    to be the directed path from $u$ to $x$.
    We also define
    \[U_x^\infty=\bigcup^\infty_{i=0}U_x^i.\] 

    \begin{cla}\label{U_x^infty}
        Let $x \in V_{bad}(C)$. 
        Let $D \in \mathcal{F}_{AB}$ such that $x$ is bad or $A$-type w.r.t.\ $D$ and let $u \in B \cap V(D)$ with $ux \in E(G)$. 
        Then each of the followings holds:
        \begin{enumerate}[$(1)$]
            \item For any $v \in U_x^\infty$, $uv \notin E(G)$ and $u^+v \notin E(G)$.
            \item For every $y \in V(C) \setminus V_{bad}(C)$ such that $y$ is adjacent to some vertex $v \in U_x^\infty$, $yu^+ \in E(G)$. 
        \end{enumerate}
    \end{cla}
    \begin{proof}
        We prove (1) and (2) simultaneously by applying induction on $i$ with $v \in U^i_x$. 
        For $i=0$, we consider $v=x^+$ as $U_x^0=\{x^+\}$. 
        Since $ux \in E(G)$, we have $u^+x^+ \notin E(G)$. 
        Furthermore, since $u$ is of $B$-type, we have $ux^+ \notin E(G)$. 
        For every $y \in V(C) \setminus V_{bad}(C)$ such that $x^+y \in E(G)$, we have $yu \in E(G)$ or $yu^+ \in E(G)$ by considering two edges $x^+y$ and $uu^+$. 
        Since $y \notin V_{bad}(C)$, $yu \notin E(G)$ by Claim \ref{bad_set_lemma}. Thus, we have $yu^+ \in E(G)$.

        Assume now that both (1) and (2) are true for every $j=0,1,\ldots,i-1$ with $i \geq 1$. 
        Let $v \in U_x^i$. 
        By the definition of $U_x^i$, there exists $w \in U_x^{i-1}$ such that $v^\dagger \notin V_{bad}(C)$ and $v^\dagger w \in E(G)$ on $P_{[w,x]}$. 
        By the induction hypothesis, $v^\dagger u^+ \in E(G)$ and $U_x^j \subset \overline{V}_C(uu^+)$ for every $j=0,1,\ldots,i-1$. 
        Thus, we have $v^\dagger \notin \bigcup_{j=0}^{i-1}U_x^j$. 
        Furthermore, we have $v \notin \bigcup_{j=0}^{i-1}U_x^j$ by the definition of $U_x^i$. 
        Since any edge on $P_{[w,x]}$ which is not an edge on $C$ has one end-vertex in $\bigcup_{j=0}^{i-1}U_x^j$, $vv^\dagger$ is an edge on $C$. 
        Thus, since $v^\dagger u^+ \in E(G)$, we have $vu \notin E(G)$ by Claim \ref{edge_lemma}. 
        If $vu^+ \in E(G)$, then 
        \[vu^+\overleftarrow{D}uxP_{[v,x]}v\]
        is a cycle combining $C$ and $D$ into a single cycle, contrary to the minimality of $|\mathcal{F}|$. 
        Thus, $vu^+ \notin E(G)$. 
        For every $y \in V(C) \setminus V_{bad}(C)$ such that $yv \in E(G)$, we have $yu \in E(G)$ or $yu^+ \in E(G)$ by considering two edges $yv$ and $uu^+$. 
        Since $y \notin V_{bad}(C)$, $yu \notin E(G)$ by Claim \ref{bad_set_lemma}. Thus, we have $yu^+ \in E(G)$. 
    \end{proof}

    We define 
    \[U^\infty=\bigcup_{x \in V_{bad}(C)}U_x^\infty.\]
    Note that $V_{bad}(C) \subset V_C(U^\infty)$ because $x^+ \in U^\infty$ for each $x \in V_{bad}(C)$.

    \begin{cla}\label{U^infty}
        The set $U^\infty$ is co-absorbable.
    \end{cla}
    \begin{proof}
        It suffices to show that $U_x^\infty$ is co-absorbable for a vertex $x \in V_{bad}(C)$. 
        Let $D \in \mathcal{F}_{AB}$ such that $x$ is bad or $A$-type w.r.t.\ $D$ and let $u \in B \cap V(D)$ such that $ux \in E(G)$. 
        Let $v \in U_x^\infty$. 
        If $v \in U_x^0$, then $v=x^+$ and so $v$ is co-absorbable by Claims \ref{A-type} and \ref{bad_set_co-abso}. 
        Thus, we assume that $v \in U_x^i$ for $i \geq 1$. 
        By the definition of $U_x^i$, there exists a spanning path $P_{[v,x]}$ of $C$ with end vertices $v$ and $x$. 
        By Claim \ref{U_x^infty}(1), we have $uv,uv^+ \notin E(G)$. 
        Let $y$ be the neighbor of $v$ on $P_{[v,x]}$. 
        Considering two edges $vy$ and $uu^+$, we have $yu \in E(G)$ or $yu^+ \in E(G)$. 
        Since $U_x^j \subset \overline{V}_C(uu^+)$ for every $j \leq i-1$, we have $y \notin \bigcup_{j=0}^{i-1}U_x^j$. 
        Furthermore, $v \notin \bigcup_{j=0}^{i-1}U_x^j$ by the definition of $U_x^i$. 
        Thus, $vy$ is an edge on $C$ since any edge on $P_{[v,x]}$ which is not an edge of $C$ has one end-vertex in $\bigcup_{j=0}^{i-1}U_x^j$. 
        If $y \in V_{bad}(C)$, then $v$ is co-absorbable by Claims \ref{A-type} and \ref{bad_set_co-abso}. 
        So we may assume that $y \notin V_{bad}(C)$. 
        Now we have $yu^+ \in E(G)$ by Claim \ref{U_x^infty}(2). 
        Thus, considering a cycle 
        \[y P_{[v,x]}xu \overleftarrow{D}u^+y,\] 
        we find that $v$ is co-absorbable.
    \end{proof}

    Applying Claim \ref{co-abso} to $U^\infty \cup B_0$, we find that the set $U^\infty \cup B_0$ is independent in $G$. 
    In particular, we have $U^\infty \cap V_C(U^\infty)=\emptyset$ since $U^\infty$ is independent in $G$.

    \begin{cla}\label{|U^infty|}
        $|V_C(U^\infty)| \leq 2|U^\infty|$.
    \end{cla}
    \begin{proof}
        First, we show that for each vertex $y \in V_C(U^\infty)$, there exists $v \in U^\infty$ such that $vy \in E(C)$. 
        Let $y \in V_C(U^\infty)$. If $y \in V_{bad}(C)$, then we have $y^+ \in U^\infty$. 
        So we may assume that $y \notin V_{bad}(C)$. 
        Let $x \in V_{bad}(C)$ be a vertex satisfying $y \in V_C(U_x^\infty)$. 
        If $y \in V_C(U_x^0)$, then $y^- \in U_x^1$ because $yx^+ \in E(G)$ and $y \notin V_{bad}(C)$. 
        So assume that $i \geq 1$ and $y \in V_C(U_x^i) \setminus V_C(\bigcup_{j=0}^{i-1}U_x^j)$. 
        Let $w \in U_x^i$ be a vertex with $wy \in E(G)$, and let $v$ be the predecessor of $y$ on $P_{[w,x]}$. 
        Since $V_C(U^\infty) \cap U^\infty=\emptyset$, we have $y \notin U_x^\infty$. 
        By the assumption that $y \in V_C(U_x^i) \setminus V_C(\bigcup_{j=0}^{i-1}U_x^j)$, we have $v \notin \bigcup_{j=0}^{i-1}U_x^j$. 
        Since any edge on $P_{[w,x]}$ which is not an edge of $C$ has one end-vertex in $\bigcup_{j=0}^{i-1}U_x^j$, we find that $yv$ is an edge on $C$. 
        Since $y \notin V_{bad}(C)$ and $wy \in E(G)$, we have $v \in U_x^i$ or $v \in U_x^{i+1}$.

        By the above discussion, we find that $V_C(U^\infty)=N_C(U^\infty)$. 
        Since $|N_C(U^\infty)| \leq 2|U^\infty|$, we obtain Claim \ref{|U^infty|}.
    \end{proof}

    Let $W=A_0 \cup V_C(U^\infty)$. 
    Since $V_{bad}(C) \subset V_C(U^\infty)$, every vertex in $U^\infty \cup B_0$ is an isolated vertex in $G-W$. 
    Hence $\omega(G-W) \geq |U^\infty|+|B_0| \geq |B_0| \geq 2$ and
    \[\frac{|W|}{\omega(G-W)} \leq \frac{|A_0|+|V_C(U^\infty)|}{|B_0|+|U^\infty|} \leq \frac{|A_0|+2|U^\infty|}{|B_0|+|U^\infty|} < 2\]
    by Claim \ref{|U^infty|}, which contradicts the toughness condition of $G$.

    \begin{case}
        $|\overline{\mathcal{F}}_{AB}|=2$.
    \end{case}

    Let $C$ and $D$ be the cycles in $\overline{\mathcal{F}}_{AB}$, and let $uu^+ \in E(C)$ and $vv^+ \in E(D)$ be fixed $B$-type edges. 
    We then define
    \[X_C=\overline{V}_C(vv^+), \ Y_C=V_C(vv^+), \ X_D=\overline{V}_D(uu^+) \text{ and }Y_D=V_D(uu^+).\] 
    Note that by Observation \ref{prop:2K_2-free}, $X_C$ and $X_D$ are independent sets in $G$.
    Also, by Claim \ref{B_edge}, $C$ and $D$ are $X_CY_C$-alternating and $X_DY_D$-alternating, respectively.
    In particular, we have $|X_C|=|Y_C|=\frac{1}{2}|V(C)|$, $|X_D|=|Y_D|=\frac{1}{2}|V(D)|$ and $|A_0|=|B_0|=\frac{1}{2}(n-|V(C)|-|V(D)|)$.

    \begin{cla}\label{XY-alt_lem}
        Let $H \in \overline{\mathcal{F}}_{AB}$. 
        If $V_{bad}(H) \cap X_H \neq \emptyset$, then $Y_H$ is co-absorbable.
    \end{cla}
    \begin{proof}
        Suppose that $V_{bad}(H) \cap X_H \neq \emptyset$, and take $x \in V_{bad}(H) \cap X_H$. 
        Let $k=|X_H|=|Y_H|$ and
        \[y_1=x^+,y_2=y_1^{++},y_3=y_2^{++},\ldots,y_k=y_{k-1}^{++}=x^-,\]
        so that $Y_H=\{y_1,y_2, \ldots, y_k\}$. 
        To prove Claim \ref{XY-alt_lem}, we show that for each $i \in \{1,\ldots,k-1\}$, the following statement $\mathcal{P}(i)$ is true by the induction on $k-i$.

        \noindent $\bm{\mathcal{P}(i)}$: There exists $s,t \in \{1,2,\ldots,k\}$ with $t-s=i$ such that each of the following holds.
        \begin{enumerate}[$(1)$]
            \item $\{y_1,y_2,\ldots,y_s\}$ and $\{y_t,y_{t+1},\ldots,y_k\}$ are co-absorbable;
            \item there exists an $xy_s$-path $P_{s,t}$ in $G$ such that $P_{s,t}$ consists of all vertices on $H$ and $y_s\overrightarrow{H}y_t$ is a subpath of $P_{s,t}$;
            \item there exists an $xy_t$-path $Q_{s,t}$ in $G$ such that $Q_{s,t}$ consists of all vertices on $H$ and $y_s \overrightarrow{H}y_t$ is a subpath of $Q_{s,t}$.
        \end{enumerate}

        First, $\mathcal{P}(k-1)$ is true by $s=1$ and $t=k$. 
        Indeed, $y_1=x^+$ and $y_k=x^-$ are co-absorbable by Claim \ref{bad_set_co-abso}. 
        Moreover, $P_{1,k}=x\overleftarrow{H}y_1$ and $Q_{1,k}=x \overrightarrow{H}y_k$ are the required paths. 

        Now assume that $\mathcal{P}(i+1)$ is true for $1 \leq i <k-1$. 
        By the induction hypothesis, there exist two integers $s^*,t^* \in \{1,2,\ldots,k\}$ and two paths $P_{s^*,t^*}$, $Q_{s^*,t^*}$ satisfying the conditions of $\mathcal{P}(i+1)$. 
        To show $\mathcal{P}(i)$, we focus on two edges $y_{s^*}y_{s^*}^+$ and $y_{t^*}y_{t^*}^-$. 
        As $t^*-s^*=i+1 \geq 2$, we have $y_{s^*}^+ \neq y_{t^*}^-$. 
        Since $y_{s^*}$ and $y_{t^*}$ are co-absorbable, $y_{s^*}y_{t^*} \notin E(G)$. 
        Moreover, we have $y_{s^*}^+y_{t^*}^- \notin E(G)$ since $y_{s^*}^+,y_{t^*}^- \in X_H$ (recall that $X_H$ is independent in $G$). 
        Thus, $y_{s^*}y_{t^*}^- \in E(G)$ or $y_{s^*}^+y_{t^*} \in E(G)$ since $G$ is $2K_2$-free.

        If $y_{s^*}y_{t^*}^- \in E(G)$, we take 
        \[s=s^*, \ t=t^*-1, \ P_{s,t}=P_{s^*,t^*} \text{ and }Q_{s,t}=xP_{s^*,t^*}y_{t^*}^-y_{s^*}P_{s^*,t^*}y_t.\]
        Then, we can see that $y_s\overrightarrow{H}y_t$ is a subpath of both $P_{s,t}$ and $Q_{s,t}$. 
        Moreover, $\{y_1,\ldots, y_s\}$ and $\{y_{t+1}, \ldots, y_k\}$ are co-absorbable by the induction hypothesis. 
        We now prove that $y_t$ is co-absorbable. 
        If $x$ is of $A$-type, then we take a cycle $K \in \mathcal{F} \setminus \{H\}$ and a vertex $z \in V(K)$ such that $xz,xz^+ \in E(G)$. 
        If $y_tz \in E(G)$ or $y_tz^+ \in E(G)$, then
        \begin{eqnarray*}
            \left \{
            \begin{aligned}
                &xQ_{s,t}y_tz \overleftarrow{K}z^+x &\text{ if }& y_tz \in E(G)\\
                &xQ_{s,t}y_tz^+ \overrightarrow{K}zx &\text{ if }& y_tz^+ \in E(G)\\
            \end{aligned}
            \right.
        \end{eqnarray*}
        is a cycle of $G$ on $V(H) \cup V(K)$, a contradiction to the minimality of $|\mathcal{F}|$. 
        Hence we have $y_tz, y_tz^+ \notin E(G)$. 
        Considering the edges $y_t^-y_t$ and $zz^+$, we have $y_t^-z \in E(G)$ or $y_t^-z^+ \in E(G)$ since $G$ is $2K_2$-free. 
        Then,
        \begin{eqnarray*}
            \left \{
            \begin{aligned}
                &xQ_{s,t}y_t^-z \overleftarrow{K}z^+x &\text{ if }& y_t^-z \in E(G)\\
                &xQ_{s,t}y_t^-z^+ \overrightarrow{K}zx &\text{ if }& y_t^-z^+ \in E(G)\\
            \end{aligned}
            \right.
        \end{eqnarray*}
        is a cycle of $G$ on $(V(H) \cup V(K)) \setminus \{y_t\}$. 
        Thus, $y_t$ is co-absorbable. 
        If $x$ is not of $A$-type, then we take a cycle $K \in \mathcal{F}_{AB}$ such that $x$ is bad w.r.t.\ $K$. 
        Let $z \in V(K) \cap B$. 
        If $y_tz \in E(G)$, then $y_t$ is in $V_{bad}(H)$ by Claim \ref{bad_set_lemma}, and so $y_t^+=y_{t^*}^-$ is co-absorbable by Claim \ref{bad_set_co-abso}. 
        However, this contradicts Claim \ref{co-abso} because $y_{t^*}$ is co-absorbable. 
        Thus, we have $y_tz \notin E(G)$. 
        If $y_tz^+ \in E(G)$, then 
        \[xQ_{s,t}y_tz^+ \overrightarrow{K}zx\]
        is a cycle of $G$ on $V(H) \cup V(K)$, a contradiction to the minimality of $\mathcal{F}$. 
        Thus, we have $y_tz^+ \notin E(G)$. 
        Considering the edges $y_t^-y_t$ and $zz^+$, we have $y_t^-z \in E(G)$ or $y_t^-z^+ \in E(G)$ since $G$ is $2K_2$-free. 
        If $y_t^-z \in E(G)$, then we have $y_t^- \in V_{bad}(H)$ by Claim \ref{bad_set_lemma}. 
        Thus, $y_t$ is co-absorbable by Claim \ref{bad_set_co-abso}. 
        If $y_t^-z^+ \in E(G)$, then 
        \[xQ_{s,t}y_t^-z^+ \overrightarrow{K}zx\] 
        is a cycle of $G$ on $(V(H) \cup V(K)) \setminus \{y_t\}$. 
        Thus, we find that $y_t$ is co-absorbable.

        Next suppose that $y_{s^*}^+y_{t^*} \in E(G)$. 
        In this case, we take
        \[s=s^*+1, \ t=t^*, \ P_{s,t}=xQ_{s^*,t^*}y_{s^*}^+y_{t^*}Q_{s^*,t^*}y_s\text{ and }Q_{s,t}=Q_{s^*,t^*}.\]
        We can similarly show that they are required ones in $\mathcal{P}(i)$. 
        We finally conclude that $\mathcal{P}(1)$ is true. 
        In particular, we find that $Y_H$ is co-absorbable.
    \end{proof}

    \begin{cla}\label{V_bad and X are disjoint}
        $X_C \cap V_{bad}(C)=X_D \cap V_{bad}(D)=\emptyset$. Thus, $X_C \cup B_0$ and $X_D \cup B_0$ are independent sets in $G$.
    \end{cla}
    \begin{proof}
        If $X_C \cap V_{bad}(C)$ and $X_D \cap V_{bad}(D)$ are not empty, then $Y_C \cup Y_D \cup B_0$ is co-absorbable by Claims \ref{A-type} and \ref{XY-alt_lem}. 
        However, by Claim \ref{co-abso}, this is an independent set in $G$ of order $\frac{n}{2}$, which contradicts $\alpha(G) \leq \frac{n}{3}$. 
        Therefore, we may assume without loss of generality, that $X_C \cap V_{bad}(C)=\emptyset$. 
        By Claim \ref{bad_set_lemma}, $X_C \cup B_0$ is independent in $G$.

        Suppose that $X_D \cap V_{bad}(D) \neq \emptyset$. Then $Y_D \cup B_0$ is co-absorbable by Claims \ref{A-type} and \ref{XY-alt_lem}. If $E_G(X_C,Y_D)=\emptyset$, then $X_C \cup Y_D \cup B_0$ is an independent set in $G$ of order $\frac{n}{2}$, contrary to $\alpha(G) \leq \frac{n}{3}$. 
        Hence $E_G(X_C,Y_D) \neq \emptyset$. 
        Let $x \in X_C$ and $y \in Y_D$ be two vertices with $xy \in E(G)$. 
        Since $X_C \cap A=\emptyset$, $x$ is of $B$-type. 
        Thus, we have $xy^{++} \in E(G)$ by Claim \ref{B-type-co-abso}. 
        By the repetition of this argument, we find that $x$ is adjacent to all vertices in $Y_D$. 
        However, since $Y_D \cap \{v,v^+\} \neq \emptyset$, $x$ is adjacent to $v$ or $v^+$, contrary to the definition of $X_C$. Thus, we find that $X_D \cap V_{bad}(D)$ is empty.
    \end{proof}

    \begin{cla}\label{X is adjacent to B-edge}
        There exist $B$-type edges $u_0u_0^+ \in E(C)$ and $v_0v_0^+ \in E(D)$ such that 
        \[V_C(v_0v_0^+)=X_C, \ \overline{V}_C(v_0v_0^+)=Y_C, \ V_D(u_0u_0^+)=X_D \text{ and } \overline{V}_D(u_0u_0^+)=Y_D.\]
    \end{cla}
    \begin{proof}
        If $E_G(X_C,X_D)=\emptyset$, then by Claim \ref{V_bad and X are disjoint}, $X_C \cup X_D \cup B_0$ is an independent set in $G$ of order $\frac{n}{2}$, contrary to $\alpha(G) \leq \frac{n}{3}$.
        Thus, we have $E_G(X_C,X_D) \neq \emptyset$.
        Let $u_0 \in X_C$ and $v_0 \in X_D$ be two vertices with $u_0v_0 \in E(G)$.
        If $u_0^+$ is of $A$-type, then $u$ is co-absorbable by Claim \ref{A-type}. 
        Since $X_D \cap A=\emptyset$, $v_0$ is of $B$-type. 
        Thus, we have $u_0^{++}v_0 \in E(G)$ by Claim \ref{B-type-co-abso}. 
        Since the cycle $C$ is not $AB$-alternating, we can choose $u_0 \in X_C$ such that $u_0v_0 \in E(G)$ and $u_0^+$ is of $B$-type.
        Since $u_0u_0^+$ is of $B$-type and $v_0 \in X_D \cap V_D(u_0u_0^+)$, we find that $X_D=V_D(u_0u_0^+)$ and $Y_D=\overline{V}_D(u_0u_0^+)$ by Claim \ref{B_edge}. 
        We can similarly find a required $B$-type edge $v_0v_0^+ \in E(D)$. 
    \end{proof}

    By Claim \ref{X is adjacent to B-edge}, we can take $B$-type edges $u_0u_0^+ \in E(C)$ and $v_0v_0^+ \in E(D)$ such that $V_D(u_0u_0^+)=X_D$ and $V_C(v_0v_0^+)=X_C$, respectively.
    Then we have
    \[X_C=\overline{V}_C(vv^+)=V_C(v_0v_0^+), \ Y_C=V_C(vv^+)=\overline{V}_C(v_0v_0^+),\]
    \[X_D=\overline{V}_D(uu^+)=V_D(u_0u_0^+) \text{ and } Y_D=V_D(vv^+)=\overline{V}_D(v_0v_0^+).\]
    In particular, we have $uu^+ \neq u_0u_0^+$ and $vv^+ \neq v_0v_0^+$.
    Moreover, by the symmetry of $X_C$ and $Y_C$ ($X_D$ and $Y_D$), Claims \ref{XY-alt_lem} and \ref{V_bad and X are disjoint} hold even if $X_C$ and $Y_C$ ($X_D$ and $Y_D$) are exchanged.
    Thus, we have $V_{bad}(C)=V_{bad}(D)=\emptyset$, which implies that every vertex in $V(C) \cup V(D)$ is of $B$-type.

    Since every edge on $C$ is of $B$-type, we can choose $uu^+$ and $u_0u_0^+$ so that $u^+=u_0$. 
    Then any vertex on $D$ is not adjacent to $u_0$ since $X_D=\overline{V}_D(uu^+)$ and $Y_D=\overline{V}_D(u_0u_0^+)$. 
    However, $u_0$ is adjacent to one of $\{v,v^+,v_0,v_0^+\}$ since $u_0 \in X_C \cup Y_C=V_C(v_0v_0^+) \cup V_C(vv^+)$, a contradiction. 
    This completes the proof of Theorem \ref{thm:main2}. 
    \qed

\end{prf3}

\section*{Acknowledgements}
We would like to express our gratitude to the referees for their helpful comments.
The first author's work (K. Ota) was supported by JSPS KAKENHI Grant Number 16H03952.
The second author's work (M. Sanka) was supported by JST Doctoral Program Student Support Project (JPMJSP2123).

\end{document}